\documentclass[12pt, leqno]{article}
\usepackage{amssymb,amsmath,amsfonts,amsthm}%,amsxtra}%,fancyhdr}
\usepackage[none]{hyphenat}
\usepackage{tikz,pdflscape}
\usetikzlibrary{matrix,arrows}
\usepackage[utf8]{inputenc}

\DeclareMathAlphabet{\chan}{T1}{pzc}{mb}{it}

\newcommand{\mc}{\mathcal}
\newcommand{\set}[1]{\{#1\}}
\newtheorem{thm}{Theorem}

\newtheorem{pbm}{Problem}

\theoremstyle{definition}
\newtheorem{example}{Example}

\newenvironment{Proof}{\begin{trivlist} \item[] {\bf Proof.}}{\hspace*{0pt}\hfill $\Box$\end{trivlist}}
\newenvironment{iproof}{\begin{trivlist} \item[] {\quad}}{\hspace*{0pt}\hfill $\Box$\end{trivlist}}

\newsavebox{\Theoremtype}
\newsavebox{\Theoremlabel}
\newenvironment{proofref}[2]
{\sbox{\Theoremtype}{\textbf{#1}}
 \sbox{\Theoremlabel}{\textbf{\ref{#2}}}
 \begin{trivlist} \item[] {\bf Proof of \usebox{\Theoremtype} \usebox{\Theoremlabel}.}\rm}
{\hspace*{0pt}\hfill$\Box$\end{trivlist}}

\newtheoremstyle{break}% an optional reference is inserted after the theorem number
{\topsep}%	Space above
{\topsep}%	Space below
{\it}%         		Body font
{}%			Indent amount (empty = no indent, \parindent = para indent)
{}%			Thm head font
{}%        		Punctuation after thm head
{ }%	Space after thm head: " " = normal interword space;	\newline = linebreak
{\thmname{\textbf{#1}}\thmnumber{ \textbf{#2.}}}% Thm head spec (can be left empty, meaning `normal')

\theoremstyle{break}

\newtheoremstyle{ref}% an optional reference is inserted after the theorem number
{\topsep}	%      Space above
{\topsep}	%      Space below
{\it}%         		Body font
{}%         		Indent amount (empty = no indent, \parindent = para indent)
{}%			Thm head font
{}%        		Punctuation after thm head
{ }%			Space after thm head: " " = normal interword space;
{\thmname{{\bfseries#1}}\thmnumber{ \bfseries#2\thmnote{\rm #3}\bfseries .}}%Thm head spec (can be left empty, meaning `normal')

\theoremstyle{ref}
\newtheorem{lem}[thm]{Lemma}

\newtheorem{cor}[thm]{Corollary}

\newtheoremstyle{nnref}% an optional reference is inserted after the theorem number
{\topsep}	%      Space above
{\topsep}	%      Space below
{\it}%         		Body font
{}%         		Indent amount (empty = no indent, \parindent = para indent)
{}%			Thm head font
{}%        		Punctuation after thm head
{ }%			Space after thm head: " " = normal interword space;
{\thmname{\textbf{#1}\thmnote{\textrm{ #3}}\textbf{.}}}%Thm head spec (can be left empty, meaning `normal')

\theoremstyle{nnref}
\newtheorem{defn}{Definition}

\begin{document}
\sloppy
%\pagenumbering{arabic}
%\pagestyle{fancyplain}
%\fancyhead[CO]{\large\textbf{NOT TO BE DUPLICATED OR DISTRIBUTED WITHOUT PERMISSION}}
\title{Productively Lindelöf spaces may all be $D$}
\author{Franklin D. Tall\makebox[0cm][l]{$^1$}}

\footnotetext[1]{Research supported by grant A-7354 of the Natural Sciences and Engineering Research Council of Canada.\vspace*{1pt}}
%\date{\today}
\maketitle

\begin{abstract}
We give easy proofs that a) the Continuum Hypothesis implies that if
the product of $X$ with every Lindelöf space is Lindelöf, then $X$ is
a $D$-space, and b) Borel's Conjecture implies every Rothberger space
is Hurewicz.
\end{abstract}

\renewcommand{\thefootnote}{}
\footnote
{\parbox[1.8em]{\linewidth}{$(2010)$ Mathematics Subject Classification. 54D20, 54B10, 54D55; Secondary 54A20, 03F50.}\vspace*{3pt}}
\renewcommand{\thefootnote}{}
\footnote
{\parbox[1.8em]{\linewidth}{Key words and phrases: Productively Lindel\"of,  $D$-space, projectively $\sigma$-compact, Menger, Hurewicz.}}

\section{Introduction}

\begin{defn}
A topological space is a \emph{$D$-space} if for every assignment $f$ from
points to open neighborhoods of them, there is a closed discrete $D \subseteq
X$ such that $\{ f(x) : x \in D \}$ covers $X$.
\end{defn}

$D$-spaces are currently a hot topic in set-theoretic topology --- see
the two recent surveys \cite{Eisworth2007}, \cite{Gruenhage2009}. For
the non-specialist, observe that a $T_1$ space is compact if and only
if it is a countably compact $D$-space. The primary question of interest
is whether every Lindelöf space is a $D$-space
\cite{vanDouwenPfeffer}. We shall assume all spaces are $T_3$.

\emph{Productively Lindelöf} spaces, i.e. spaces such that their product
with every Lindelöf space is Lindelöf, have been studied in connection with
two classic problems of E. A. Michael:

\begin{pbm}
Is the product of a Lindelöf space with the space of irrationals Lindelöf?
\end{pbm}

\begin{pbm}
If $X$ is productively Lindelöf, is $X^\omega$ Lindelöf (we
say $X$ is \emph{powerfully Lindelöf})?
\end{pbm}

For an extensive list of references concerning these problems see
\cite{TallMenger}. The primary result of note is due to Michael,
being implicitly proved in \cite{Michael1971}. It is explicitly stated
and proved in \cite{Alster1987}:

\begin{lem}\label{lem1}
The Continuum Hypothesis implies that productively Lindelöf
metrizable spaces are $\sigma$-compact.
\end{lem}

Our tools include selection principles and topological games. As a byproduct,
we obtain an easy proof of the consistency of every Rothberger space being
Hurewicz (see definitions below).

Section \ref{sec2} gives a self-contained easy proof of the result of
the title. Sections \ref{sec3} and \ref{sec4} are more for specialists, varying
the themes of Section \ref{sec2}. Section \ref{sec5} contains a short proof
of b) of the abstract. Section \ref{sec6} is built around a diagram of the
relationships among the properties we have discussed.

\section{$CH$ implies productively Lindelöf spaces are $D$}\label{sec2}

We shall give a short, reasonably elementary proof that:

\begin{thm}\label{thm2}
The Continuum Hypothesis implies productively Lindelöf spaces are $D$.
\end{thm}

We shall prove Theorem \ref{thm2} by combining Lemma \ref{lem1}
with results of
Arhangel'ski\u\i{} \cite{Arhangelskii2000} and Aurichi \cite{Aurichi}. Theorem
\ref{thm2} is a considerable improvement over \cite{AurichiTall} and
\cite{TallNote}, in which additional assumptions of separability, first
countability, or sequentiality were required.

We require two definitions:

\begin{defn}[\cite{Arhangelskii2000}]
A space is \emph{projectively $\sigma$-compact} if its continuous image in
any separable metrizable space is $\sigma$-compact.
\end{defn}

\begin{defn}
A space $X$ is \emph{Menger} if whenever $\{ \mathcal{U}_n \}_{n < \omega}$
are open covers of $X$, there are finite subsets $\mathcal{V}_n$ of
$\mathcal{U}_n$, $n < \omega$, such that $\{ \bigcup \mathcal{V}_n :
n < \omega \}$ is an open cover.
\end{defn}

This latter concept was introduced by Hurewicz \cite{Hurewicz1925} and
has been studied under various names since then. In particular, some
confusion arises because Arhangel'ski\u\i{} calls this property
\emph{Hurewicz}. However, our terminology is generally accepted. A
breakthrough on the Lindelöf $D$-problem occurred when Aurichi
\cite{Aurichi} proved:

\begin{lem}\label{lem3}
Menger spaces are $D$.
\end{lem}

Combining this with Arhangel'ski\u\i{}'s

\begin{lem}\label{lem4}
Projectively $\sigma$-compact Lindel\"{o}f spaces are Menger.
\end{lem}

\noindent
we need only establish that productively Lindelöf spaces are projectively
$\sigma$-compact. But this follows quickly from Lemma \ref{lem1}, since
continuous images of productively Lindelöf spaces are easily seen to be
productively Lindelöf.
\qed\medskip

For the convenience of the reader, we sketch the proofs of Lemmas \ref{lem1},
\ref{lem3} and \ref{lem4}.

\begin{proofref}{Lemma}{lem1}
Embed $X$ in $[0,1]^{\aleph_0}$. $[0,1]^{\aleph_0}$ has a countable
base, so by $CH$ we can take open subsets $\{ U_\alpha \}_{\alpha <
\omega_1}$ of $[0,1]^{\aleph_0}$ such that every open set about $Y =
[0,1]^{\aleph_0} - X$ includes some $U_\alpha$. By taking countable
intersections, we can find a decreasing sequence $\{G_\beta\}_{\beta <
\omega_1}$ of $G_\delta$'s about $Y$, such that every open set about $Y$
includes some $G_\beta$. If $X$ is not $\sigma$-compact, we can assume the
$G_\beta$'s are strictly descending. Pick $p_\beta \in (G_{\beta+1} -
G_\beta) \cap X$. Put a topology on $Z = Y \cup \{ p_\beta : \beta < \omega_1\}$
by strengthening the subspace topology to make each $\{ p_\beta \}$ open.
Then $Z$ is Lindelöf, but $X \times Z$ is not, since $\{ \langle p_\beta,
p_\beta \rangle : \beta < \omega_1 \}$ is closed discrete.
\end{proofref}

\begin{proofref}{Lemma}{lem4}
  By 5.1.J(e) of Engelking's text \cite{Engelking1989}, given a
  Lindelöf space $X$, and an open cover $\mathcal{U}$, there is a
  continuous $f : X \to Y$, $Y$ separable metrizable and an open cover
  $\mathcal{V}$ of $Y$ such that $\{ f^{-1}(V) : V \in \mathcal{V} \}$
  refines $\mathcal{U}$. Given a sequence $\{ \mathcal{U}_n \}_{n <
    \omega}$ of such covers, find the corresponding $f_n$'s, $Y_n$'s
  and $\mathcal{V}_n$'s. Then the diagonal product of the $f_n$'s maps
  $X$ onto a subspace $\hat{Y}$ of $\prod Y_n$. $\hat{Y}$ is
  $\sigma$-compact, hence Menger, so we can take
  finite subsets of the $\mathcal{V}_n$'s forming a cover and then
  pull them back to $X$ to find the required
  finite subsets of the $\mathcal{U}_n$'s.
\end{proofref}

\begin{proofref}{Lemma}{lem3} (Taken from \cite{Gruenhage2009}.)
  Suppose $X$ is Menger and $f$ is a neighborhood assignment for
  $X$. We play a game in which ONE chooses in the $n$th inning an open
  cover $\mathcal{U}_n$ and TWO choses a finite $\mathcal{V}_n
  \subseteq \mathcal{U}_n$. TWO wins if $\{\bigcup \mathcal{V}_n : n <
  \omega \}$ covers $X$. Hurewicz \cite{Hurewicz1925} proved $X$ is Menger
  if and only if ONE has no winning strategy.

  ONE starts by playing $\{ f(x) : x \in X \}$. TWO responds with $\{
  f(x) : x \in S_0 \}$. ONE then plays $\{ f(x) : x \in S_0 \cup S : S$
  a finite subset of $X$, $S \cap \bigcup\{f(x) : x \in S_0 \} =
  \emptyset \}$. If TWO replies with $\{ f(x) : x \in S_0 \cup S_1
  \}$, ONE plays
  \[ \{ f(x) : x \in S_0 \cup S_1 \cup S : S \cap \bigcup \{ f(x) : x
  \in S_0 \cup S_1 \} = \emptyset \}, \] etc. This defines a strategy
  for ONE. Since $X$ is Menger, this is not a winning strategy. Let
  $S_0, \ldots, S_n, \ldots$ be the plays of TWO demonstrating
  this. Then $\bigcup_{n < \omega} S_n$ is closed discrete, and
  $\bigcup \{ f(x) : x \in \bigcup_{n<\omega} {S_n} \}$ covers $X$.
\end{proofref}

\section{Variations on the theme}\label{sec3}

We now move on to more specialized results. Since finite powers
of productively Lindelöf spaces are productively Lindelöf, we note
that:

\begin{thm}
  The Continuum Hypothesis implies that all finite powers of a
  productively Lindelöf space are Menger and hence $D$.
\end{thm}

\begin{defn}
A $\mathbf{\gamma}$\textbf{-cover} of a space is a countably infinite open cover such that each point is in all but finitely many members of the cover.  A space is \textbf{Hurewicz} if given a sequence $\{\mc{U}_n: n \in \omega\}$ of $\gamma$-covers, there is for each $n$ a finite $\mc{V}_n\subseteq\mc{U}_n$, such that either $\{\bigcup\mc{V}_n : n \in \omega\}$ is a $\gamma$-cover, or else for some $n$, $\bigcup\mc{V}_n$ is a cover.
\end{defn}

This property was also introduced in \cite{Hurewicz1925}. It falls
strictly between ``Menger'' and ``$\sigma$-compact''. Our results can
be improved to obtain:

\begin{thm}\label{thm6}
The Continuum Hypothesis implies finite powers of productively Lindelöf
spaces are Hurewicz.
\end{thm}

The proof of Theorem \ref{thm6} is a straightforward modification of what we
have done for Menger.

\begin{pbm}
Are any of our uses of the Continuum Hypothesis necessary?
\end{pbm}

The assumption of the Continuum Hypothesis in our results can be weakened
somewhat. $\mathfrak{b}$ is the least cardinal of a subset $B$ of $^\omega
\omega$ which is unbounded under eventual dominance. $CH$ implies
$\mathfrak{b} = \aleph_1$.

\begin{thm}
$\mathfrak{b} = \aleph_1$ implies every productively Lindelöf space is Menger
and hence $D$.
\end{thm}

\begin{proof}
Making the obvious definition of \emph{projectively Menger}, we see that
what Arhangel'ski\u\i{} really proved above was that \emph{Lindelöf
projectively Menger spaces are Menger}, which indeed was later proved
specifically in \cite{BonCamMat}. Thus our result follows, since Alas et
al \cite{AAJT} proved \emph{$\mathfrak{b} = \aleph_1$ implies productively
Lindelöf metrizable spaces are Menger}.
\end{proof}

\begin{cor}
Every productively Lindelöf space which is the union of $\le \aleph_1$
compact sets is Menger and hence $D$.
\end{cor}

\begin{proof}
This is proved for spaces of countable type, hence in particular for
metrizable spaces in \cite{AAJT}. Our result follows, since if a space
is the union of $\aleph_1$ compact sets, so is its continuous image.
\end{proof}

We can remove $CH$ from Theorem \ref{thm2} by strengthening the
hypothesis.  In \cite{AurichiTall} we defined a space to be
\emph{indestructibly productively Lindelöf} if it remained
productively Lindelöf in any countably closed forcing extension.

\begin{thm}
Indestructibly productively Lindelöf spaces are projectively $\sigma$-compact
and hence Hurewicz, Menger and $D$.
\end{thm}

\begin{proof}
Let $f : X \to Y$, $Y$ separable metrizable. Collapse $\max(w(X)$, $\left|
X \right|$, $2^{\aleph_0})$ to $\aleph_1$ by countably closed forcing. In the
extension, $X$ is productively Lindelöf, $Y$ is separable metrizable, and $f$
is continuous. Therefore $Y$ is $\sigma$-compact. Countably closed forcing
adds no new closed sets to separable metrizable spaces, so $Y =
\bigcup_{n<\omega} F_n$, where the $F_n$'s are in the ground model. The $F_n$'s
are countably compact in the ground model, and so they are in fact compact
there. No new countable decompositions of $Y$ are added by the forcing, so
indeed $Y$ is $\sigma$-compact in the ground model.
\end{proof}

We had earlier \cite{TallNote} obtained the Hurewicz, etc. conclusions, but this
new result is stronger.

Similarly, in \cite{AurichiTall} we proved that $\mathfrak{d} = \aleph_1$
implied productively Lindelöf metrizable spaces are Hurewicz, so:

\begin{thm}
$\mathfrak{d} = \aleph_1$ implies productively Lindelöf spaces are Hurewicz.
\end{thm}

\begin{cor}
Every productively Lindelöf space which is the union of $\le \aleph_1$
compact sets is Hurewicz.
\end{cor}

\begin{iproof}
The corollary follows since it was proved from $\mathfrak{d} > \aleph_1$ in
\cite{TallNote}.
\end{iproof}

A finer analysis leads to:

\begin{cor}
Every productively Lindelöf space which is the union of $\le \aleph_1$
compact sets is projectively $\sigma$-compact.
\end{cor}

\begin{proof}
This follows immediately from the fact that:

\begin{lem}[ \cite{AAJT}]
Every productively Lindelöf space of countable type which is the union of
$\le \aleph_1$ compact sets is $\sigma$-compact.
\end{lem}

\noindent
since every metrizable space is of countable type.
\end{proof}

Also in \cite{TallNote}, we proved that $Add(\mathcal{M}) = 2^{\aleph_0}$
implies productively Lindelöf metrizable spaces are Hurewicz. Recall
$Add(\mathcal{M})$ is the least $\kappa$ such that there are $\kappa$ many
first category subsets of $\mathbb{R}$ with union not of first category.

By the usual reasoning, we have that:

\begin{thm}
$Add(\mathcal{M}) = 2^{\aleph_0}$ implies productively Lindelöf spaces
are Hurewicz.
\end{thm}

\begin{defn}[\cite{Hansell}]
A \emph{Michael space} is a Lindelöf space such that its product with
$\mathbb{P}$, the space of irrationals, is not Lindelöf. A space is
\emph{$K$-analytic} if it is the continuous image of a Lindelöf
\v{C}ech-complete space.
\end{defn}

In \cite{TallNote} we asked whether it is consistent that every productively
Lindelöf $K$-analytic space is $\sigma$-compact. We now know this holds under
$CH$, but we can considerably weaken that hypothesis and still get that such
spaces are projectively $\sigma$-compact (and hence $D$, etc.):

\begin{thm}
If there is a Michael space, then productively Lindelöf $K$-analytic spaces
are projectively $\sigma$-compact.
\end{thm}

\begin{proof}
Let $X$ be Lindelöf \v{C}ech-complete, $g$ map $X$ onto $Y$, $Y$ productively
Lindelöf, $f$ map $Y$ onto a separable, metrizable $Z$. Then $Z$ is
$K$-analytic. Then $Z$ is analytic, since $K$-analytic subspaces of separable
metrizable spaces are analytic --- see e.g. Theorems 2.1(f) and 3.1(d) of
\cite{Hansell}. But in \cite{TallMenger} I proved:

\begin{lem}
If there is a Michael space, then productively Lindelöf analytic metrizable
spaces are $\sigma$-compact.
\end{lem}
\end{proof}

There is a Michael space if either $\mathfrak{b} = \aleph_1$ \cite{Lawrence1990}
or $\mathfrak{d} = cov(\mathcal{M})$ \cite{Moore1999}.

\section{Playing with projectively $\sigma$-compact spaces}\label{sec4}

In this section, we assume some acquaintance with topological games as
in \cite{Scheepers2004}. The players will be ONE and TWO, the games
will be of length $\omega$, and \emph{strategies} are perfect information
strategies.

Telgársky \cite{Telgarsky} proved:

\begin{lem}\label{lem9}
A metrizable space is $\sigma$-compact if and only if TWO has a winning strategy
in the Menger game for $X$.
\end{lem}

We defined the Menger game above, in the process of proving Lemma \ref{lem3}.
Scheepers \cite{Scheepers1995}
provided a more accessible proof of Lemma \ref{lem9}, noting that
metrizability was only needed in the proof for the backward implication.

We had conjectured in an earlier version of this note that metrizability was essential, but Banakh and Zdomskyy \cite{BZ} proved that it could be weakened to ``hereditarily Lindel\"{o}f".  We had also asked whether projective $\sigma$-compactness was equivalent for Lindel\"{o}f spaces to TWO having a winning strategy in the Menger game.  It isn't -- see below.  However, we can prove:

\begin{thm}
Suppose there is a winning strategy for TWO in the Menger game for $X$.  Then $X$ is projectively $\sigma$-compact.
\end{thm}

\begin{Proof}
Suppose $X$ is not projectively $\sigma$-compact.
Then there is an $f : X \to Y$ separable metrizable, such that $Y$ is not
$\sigma$-compact. Suppose there were a winning strategy for TWO in the Menger
game on $X$. We can define a strategy for the Menger game on $Y$ by simply
playing, given an open cover $\mathcal{W}$ of $Y$ (and previous information)
the finite subset $\mathcal{W}'$ of $\mathcal{W}$ such that $\{ f^{-1}(W) :
W \in \mathcal{W}' \}$ is the move of TWO for the cover $\{ f^{-1}(W) : W \in
\mathcal{W} \}$ (and the corresponding previous information). Then, since the
$\omega$-sequence of moves for $X$ would yield a cover, their images would
yield a cover of $Y$. Thus a winning strategy for $X$ entails a winning
strategy for $Y$. But $Y$ is not $\sigma$-compact so there is no such winning
strategy for it and hence none for $X$.
\end{Proof}

The diagram in Figure \ref{fig1} below shows the relationships among the
properties we have discussed in this article. A more extensive diagram
with many more Lindelöf properties can be found in
\cite{AurichiTall}, but it does not mention projective
$\sigma$-compactness, which is our main concern here. Examples showing
that implications do not reverse can be found there and below. For
projective $\sigma$-compactness, the relevant examples in addition to
Okunev's are:

\begin{example}
  A Hurewicz space which is not projectively $\sigma$-compact. Simply
  take a Hurewicz set of reals which is not $\sigma$-compact
  \cite{Just1996}. One can even get an example with finite products
  Hurewicz \cite{Tsaban2008}.
\end{example}

\begin{example}\label{ex2}
  A projectively $\sigma$-compact space which is not productively
  Lindelöf, and for which TWO does not have a winning strategy in the Menger game. J. T. Moore \cite{Moore2006} has constructed a hereditarily Lindelöf space
  $X$ such that some finite power of $X$ is not Lindelöf \cite{TZ}, but any continuous real-valued
  function on $X$ has countable range. It follows that any continuous
  function $f$ on $X$ into any separable metrizable space $Y$ has
  countable range. To see this, embed $Y$ in $[0,1]^{\aleph_0}$. If
  $f(X)$ has uncountable projection onto any factor of
  $[0,1]^{\aleph_0}$, we have a contradiction, so $f(X) \subseteq
  \prod_{n<\omega} \pi_n(f(X))$, where each factor is countable and
  hence $0$-dimensional, so $\prod_{n<\omega} \pi_n(f(X))$ is
  $0$-dimensional and hence embeds in a Cantor set included in
  $\mathbb{R}$, so indeed $f(X)$ is countable.  Now according to the Banakh-Zdomskyy result quoted above, since $M$ is hereditarily Lindel\"{o}f, if TWO had a winning strategy in the Menger game, $M$ would be $\sigma$-compact, which it isn't, since it is not productively Lindel\"{o}f.
\end{example}

\begin{example}\label{newex3}
A non-$\sigma$-compact space for which TWO has a winning strategy in the Menger game.  In \cite{Arhangelskii2000}, Arhangel'ski\v{i} gives an example -- due to Okunev -- of a projectively $\sigma$-compact space which is not $\sigma$-compact.  First, one take the Alexandrov duplicate of the space $\mathbb{P}$ of irrationals.  That is, take two copies of $\mathbb{P}$ and make one of them discrete.  A neighbourhood of $p$ in the non-discrete copy is obtained by taking a usual neighbourhood $U$ of $p$ in $\mathbb{P}$ together with $U - \set{p}$ in the discrete copy of $\mathbb{P}$.  Okunev then identifies the non-discrete copy of $\mathbb{P}$ to a point to obtain the desired space $X$, which he proves is Lindel\"{o}f and projectively $\sigma$-compact, but not $\sigma$-compact. Give the first open cover $\mathcal{U}_0$, TWO picks an element $U_0$ of $\mc{U}_0$ containing the unique non-isolated point.  Then $U_0$ is cocountable since $X$ is Lindel\"{o}f.  Now let $X - U_0 = \set{x_n : 0 < n < \omega}$.  Given the $n$th open cover, for $n > 0$, TWO picks an element containing $x_n$.  This strategy clearly wins.
\end{example}

\section{Borel's Conjecture implies Rothberger spaces are Hurewicz}\label{sec5}

As another example of the utility of projective $\sigma$-compactness,
we shall prove:

\begin{thm}\label{thm11}
Assume Borel's Conjecture. Then every Rothberger space is Hurewicz.
\end{thm}

\emph{Rothberger} is a strengthening of \emph{Menger} in that picking
\emph{one} element from each member of the sequence of open covers
suffices to yield a cover.

\begin{defn}
A set of reals $X$ has \emph{strong measure zero} if and only if given any
sequence $\{\varepsilon_n\}_{n<\omega}$, $\varepsilon>0$, $X$ can be covered
by $\{ X_n : n < \omega\}$, each $X_n$ having diameter less than
$\varepsilon_n$.
\end{defn}

Borel \cite{Borel} conjectured that every strong measure zero set is countable.
Laver \cite{Laver} proved the consistency of Borel's Conjecture. Zdomskyy
\cite{Zdomsky} proved that every paracompact Rothberger space is Hurewicz,
assuming $\mathfrak{u} < \mathfrak{g}$. We refer to his paper or to
\cite{Blass} for the
definitions of these cardinals. Scheepers and Tall \cite{ScheepersTall} observed
that paracompactness could be eased to regularity in Zdomskyy's theorem. The
hypothesis of Zdomskyy's theorem is sophisticated and the proof is non-trivial.
Our proof of Theorem \ref{thm11} is very easy:

\begin{proofref}{Theorem}{thm11}
Suppose $X$ is Rothberger. Then so is every continuous image of $X$.
Rothberger subsets of the real line have strong measure zero (see e.g.
\cite{Miller1984}) and by Borel's Conjecture are therefore countable.
By the same argument as for Example \ref{ex2}, $X$ is projectively
$\sigma$-compact. But then it is Hurewicz.
\end{proofref}

Marion Scheepers pointed out to me that Borel's Conjecture does \emph{not}
follow from $\mathfrak{u} < \mathfrak{g}$, since there is a model of Borel's
Conjecture in which $\mathfrak{b} = \aleph_1$, which implies there is an
uncountable set of reals concentrated about a countable set. Such a set is
Rothberger. See \cite{Blass} for reference to such a model.

Call a space \emph{projectively countable} if its continuous image in any
separable metrizable space is countable. We then have:

\begin{thm}[\cite{BonCamMat}]
Borel's Conjecture implies a space is Rothberger if and only if it is
Lindelöf and projectively countable.
\end{thm}

\begin{proof}
That Rothberger implies Lindelöf is obvious. We have already proved that $BC$
implies Rothberger spaces are projectively countable. The converse is proved
by the usual technique; indeed Lindelöf projectively Rothberger spaces are
Rothberger \cite{Kocinac}.
\end{proof}

\section{Implications and not}\label{sec6}

\begin{defn}[\cite{Alster1988}, \cite{Barr2007a}]
A space $X$ is \emph{Alster} if every cover $\mathcal{G}$ by $G_\delta$'s
has a countable subcover, provided that for each compact subset $K$ of $X$,
some finite subset of $\mathcal{G}$ covers $K$.
\end{defn}

Alster spaces are important in the study of Michael's problems, since they
are both productively Lindelöf and powerfully Lindelöf \cite{Alster1988}.
$\sigma$-compact spaces are Alster, but not necessarily vice versa
\cite{Alster1988}, \cite{Barr2007a}. It is not known if productively
Lindelöf spaces are Alster or even powerfully Lindelöf ---see
\cite{Alster1988}, \cite{AurichiTall}, \cite{TallNote}. Alster spaces
in which compact sets are $G_\delta$'s are $\sigma$-compact, so powerfully
Lindelöf spaces need not be Alster. $\sigma$-compact spaces as well as
Lindelöf \emph{$P$-spaces} ($G_\delta$'s are open) are Alster
\cite{Alster1988}, \cite{Barr2007a}. Alster spaces are projectively
$\sigma$-compact, but even projective countability is insufficient to imply
Alster. To see this, note that Moore's $L$-space $M$ is projectively countable
but not Alster, since some finite power of $M$ is not Lindelöf \cite{TZ}. \qed

\begin{pbm}
Does Alster imply TWO has a winning strategy in the Menger game?
Is the converse true?
\end{pbm}

We have proved or given references already for almost all of the
non-obvious implications in the diagram below. That ``indestructibly
productively Lindelöf'' implies ``powerfully Lindelöf'' is in
\cite{TallNote}. To see that Lindelöf $P$-spaces are projectively
countable, observe that if $X$ is $P$ and $Y$ has points $G_\delta$
and $f: X \to Y$, then the inverse images of points in $Y$ form a
disjoint open cover of $X$.

Moore's $L$-space is projectively countable but not Alster nor $P$ since closed
subsets are $G_\delta$'s. As mentioned, it is neither productively Lindelöf
nor powerfully Lindelöf. $2^{\omega_1}$ is compact, but it is not $P$ and is
not indestructibly productively Lindelöf \cite{AurichiTall}. A Bernstein
(totally imperfect) set of reals is powerfully Lindelöf but not productively
Lindelöf \cite{Michael1971}. See \cite{Just1996}, \cite{Tsaban2008} for
examples of sets of
reals which are Menger, but not Hurewicz, and Hurewicz but not
(projectively) $\sigma$-compact. The space of irrationals is Lindelöf $D$ but
not Menger. It is consistent that there are Rothberger spaces that are not
Hurewicz. See the discussion in Section 3 of \cite{ScheepersTall}.

\begin{landscape}
\thispagestyle{empty}
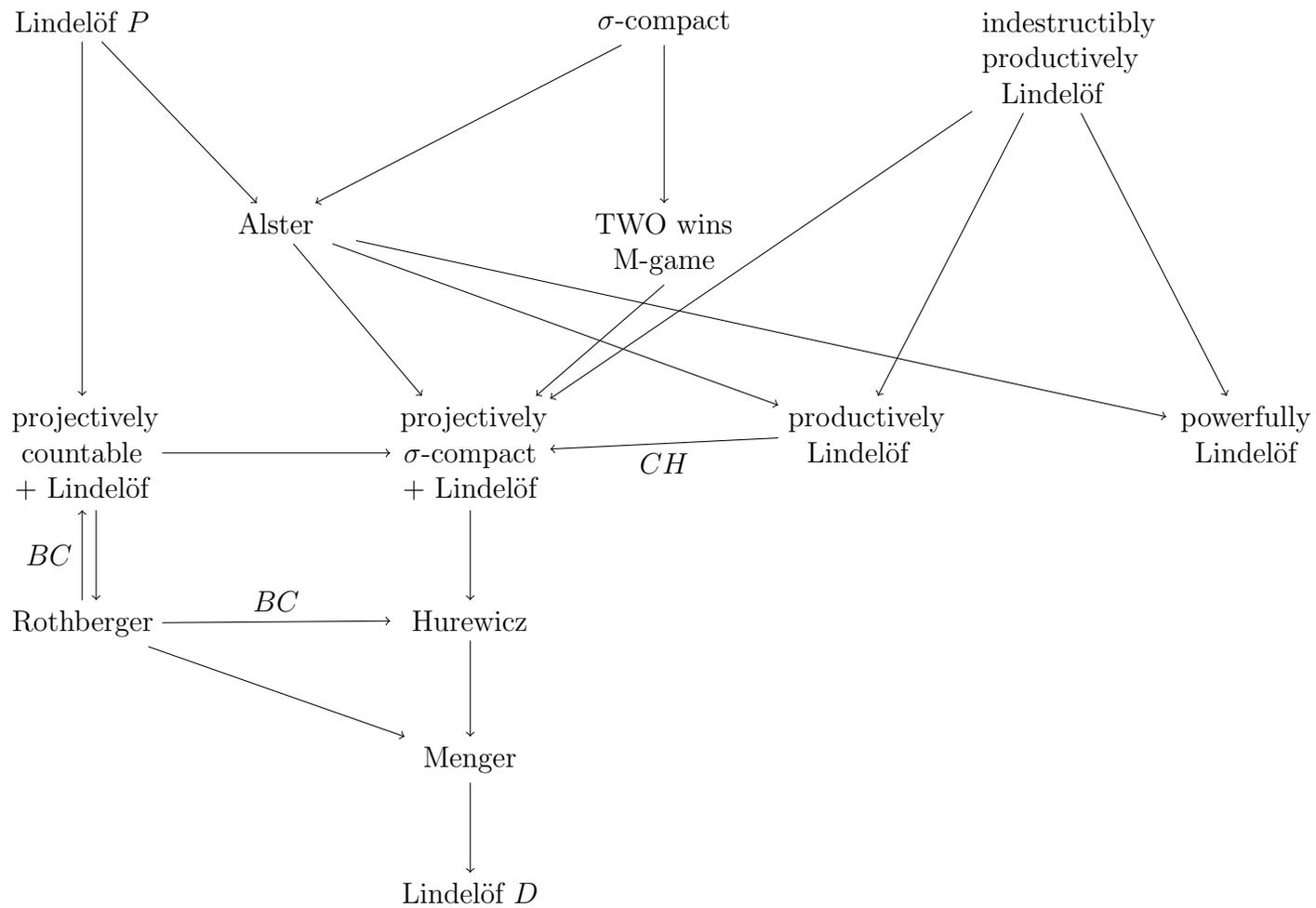
\begin{figure}
\begin{tikzpicture}
\matrix (m) [matrix of nodes,
             row sep = 13mm,
             column sep = 5mm,
             nodes={text width=2cm, text centered}]
{Lindelöf $P$ & & & $\sigma$-compact & &
indestructibly productively Lindelöf & \\
 & Alster & & TWO wins M-game & & & \\[3mm]
projectively countable + Lindelöf & & projectively $\sigma$-compact
+ Lindelöf & & productively Lindelöf & & powerfully Lindelöf \\
Rothberger & & Hurewicz & & & & \\
 & & Menger & & & & \\
 & & Lindelöf $D$ & & & & \\};
\path[->]
(m-1-1) edge (m-2-2)
        edge (m-3-1)
(m-1-4) edge (m-2-2)
        edge (m-2-4)
(m-1-6) edge (m-3-5)
        edge (m-3-7)
        edge[bend left=0] (m-3-3)
(m-2-2) edge (m-3-3)
        edge (m-3-5)
        edge[bend right=0] (m-3-7)
(m-2-4.south) edge (m-3-3)
(m-3-1) edge (m-3-3)
(m-3-5) edge node[auto]{$CH$} (m-3-3)
(m-4-1) edge node[auto]{$BC$} (m-3-1)
        edge node[auto]{$BC$} (m-4-3)
        edge (m-5-3)
(m-3-3) edge (m-4-3)
(m-4-3) edge (m-5-3)
(m-5-3) edge (m-6-3);
\path (m-3-1.south) ++ (0.2,0) coordinate (P);
\path (m-4-1.north) ++ (0.2,0) coordinate (Q);
\path[->]  (P) edge (Q);
\end{tikzpicture}
\caption{The relationships among various properties discussed.}
\label{fig1}
\end{figure}
\end{landscape}

\bibliographystyle{acm}
\bibliography{prodlindmay}

\noindent
{\rm Franklin D. Tall\\
Department of Mathematics\\
University of Toronto\\
Toronto, Ontario\\
M5S 2E4\\
CANADA\\}
\noindent
{\it e-mail address:} {\rm f.tall@utoronto.ca}

\end{document}